\documentclass[11pt]{amsart}
\usepackage[dvips]{graphicx,color}
\usepackage{amsmath,graphics}
\usepackage{amsfonts,amssymb}
\usepackage{xypic}
\theoremstyle{plain}
\newtheorem{theorem*}{Theorem}
\newtheorem*{lemma*} {Lemma}
\newtheorem{corollary*}[theorem*]{Corollary}
\newtheorem{proposition*}[theorem*]{Proposition}
\newtheorem*{conjecture*}{Conjecture}
\newtheorem{theorem}{Theorem}[section]

\newtheorem*{theorem1*}{Theorem 1}
\newtheorem*{theorem2*}{Theorem 2}
\newtheorem*{theorem3*}{Theorem 3}
\newtheorem*{theorem4*}{Theorem 4}
\newtheorem*{theorem5*}{Theorem 5}
\newtheorem*{theorem6*}{Theorem 6}
\newtheorem*{corollary1*}{Corollary 1}
\newtheorem*{corollary2*}{Corollary 2}
\newtheorem*{corollary3*}{Corollary 3}
\newtheorem*{corollary4*}{Corollary 4}
\newtheorem*{corollary5*}{Corollary 5}
\newtheorem*{corollary6*}{Corollary 6}

\theoremstyle{remark}
\newtheorem*{remark}{Remark}

\newtheorem{example*}{Example}

\theoremstyle{definition}

\def\G{\Gamma}

  \def\F{\Bbb{F}} \def\Z{\Bbb{Z}}  
    
    \def\bp{\begin{pmatrix}}
 \def\ep{\end{pmatrix}} \def\bn{\begin{enumerate}} 
   \def\en{\end{enumerate}}
\def\ba{\begin{array}} \def\ea{\end{array}}  
     \def\ti{\tilde}

\def\be{\begin{equation}} \def\ee{\end{equation}}

\def\Tor{\mbox{Tor}}

    \def\fr12{\frac{1}{2}} \def\z12{\Z[\fr12]}

\begin{document}

\title[Symplectic $4$--manifolds with $K = 0$ and the Lubotzky alternative]{Symplectic ${\bf 4}$--manifolds with $K = 0$ and the Lubotzky alternative}
\author{Stefan Friedl}
\address{Mathematisches Institut, Universit\"at zu K\"oln, Germany}
\email{sfriedl@gmail.com}
\author{Stefano Vidussi}
\address{Department of Mathematics, University of California,
Riverside, CA 92521, USA} \email{svidussi@math.ucr.edu} \thanks{The second author was partially supported by NSF grant \#0906281.}
\date{\today}

\begin{abstract}
In this paper we use the Lubotzky alternative for finitely generated linear groups to determine  which $4$-manifolds admitting a free circle action can be endowed with a
symplectic structure with trivial canonical class.
\end{abstract}

\maketitle

\section{Introduction and main results}
 The study of symplectic manifolds with $K = 0$ is a  central topic of symplectic geometry. Restricting ourselves to (real) dimension $4$, this problem (as well as the closely connected case of manifolds with Kodaira dimension $0$) has
been studied by many authors (see e.g. \cite{Ge92,Li06,Li06b,Bau08,Us09}).
In this note we will determine  which 4-manifolds admitting a free circle action can be endowed with a
symplectic structure with trivial canonical class.

In order to state our results  we have to introduce some definitions and notation. Let $M$ be a smooth $4$--manifold that admits a free circle action. Denote by $N$ the orbit space of this action, a smooth $3$--manifold. Then we can consider $M$ as the total space of a principal $S^1$-bundle over $N$, determined by the Euler class $e \in H^2(N;\Z)$.
We will refer to the Euler class of this $S^1$-bundle as the Euler class of the free $S^1$-action on $M$.

The following now lists all  examples of symplectic 4-manifolds with $K=0$ which are known to the authors (see also \cite{Li06} and \cite{Ge92}):
\bn
\item the $K3$ surface,
\item $T^2$--bundles over $T^2$,
\item $S^1$-bundles over a $T^2$--bundle $N$.
\en
(Note that the second and third classes of examples overlap.) The $K3$-surface is well-known to be symplectic; torus bundles over a torus where shown by Geiges \cite{Ge92} to be symplectic, and the third type of example
was shown to be symplectic by many authors (see e.g. \cite{Th76,Bou88,FGM91,FV10a}). In all three cases it is well-known that the canonical is zero.
It is reasonable to conjecture that this list is complete.
In this paper we will prove this conjecture for the class of 4-manifolds which admit a free circle action.
More precisely, we have the following result.
\begin{theorem*} \label{lub} Let $(M,\omega)$ be a symplectic manifold with trivial canonical class. If $M$ admits a free circle action with orbit space $N$ and Euler class $e \in H^2(N;\Z)$, then $N$ is a torus bundle over a circle and, perhaps with the exclusion of the case when $b_1(N)=1$ and $e$ is torsion,  $M$ is a torus bundle over a torus.
\end{theorem*}

\begin{remark}
\bn
\item
If $M$ is in fact a product $M=S^1\times N$ (i.e. $e = 0$), then this theorem was first proved in \cite{FV08a} and reproved in greater generality in \cite{FV08b}
(see also \cite{FV10b}).
Note though that both approaches rely heavily on the result, due to Kronheimer \cite{Kr99}, that for manifolds of the form
$S^1 \times N$ there exists a refined adjunction inequality. This inequality  allows one to
constrain the Thurston norm of a $3$--manifold in terms of the canonical class. In particular, if $M=S^1\times N$ is symplectic with $K = 0$ this
constraint translates into the fact that $N$ must have vanishing Thurston norm, in particular $N$ contains incompressible tori. This strong topological information then gives information on the homology of the finite  covers of $S^1\times N$. In \cite{FV08a,FV08b} we showed that this information is compatible with the existence of a symplectic structure with canonical equal to zero only if $N$ is a torus bundle over $S^1$.
\item Bowden \cite{Bow09}  studied symplectic 4-manifolds with a free circle action over a $3$--manifold with vanishing Thurston norm along the lines of \cite{FV08a}. This discussion shows that our main theorem contains Bowden's results
\en
\end{remark}

If $M$  is a  4-manifold with a free circle action \emph{and non-trivial Euler class}, then
 it is not known whether a refined adjunction inequality similar to that established in
\cite{Kr99} holds. We are therefore  forced to gather information by other means. We will succeed in
doing so by extending the approach of Section 2 of \cite{FV08a}, using as new topological
ingredient a consequence of the Lubotzky alternative for finitely generated linear groups.\\

\noindent
\textbf{Note.} The results of this paper previously appeared as part of a preprint by the authors with the title ``Symplectic $4$--manifolds with a free circle action".

\section{Proofs and discussion}

For sake of convenience, we split the proof of Theorem \ref{lub} in the two cases where the Euler class $e \in H^2(N)$ is nontorsion and, respectively, torsion. We will first treat the former case.
The latter case, which can be reduced via covering theory to the product case treated in \cite{FV08a}, will be discussed subsequently.

\begin{theorem} \label{betti}
Let $(M,\omega)$ be a symplectic manifold with trivial canonical class. If $M$ admits a free circle action with orbit space $N$ and nontorsion Euler class $e \in H^2(N;\Z)$, then $N$ is a torus bundle over a circle and $M$ is a torus bundle over a torus.
\end{theorem}

\begin{proof}
We denote the orbit space of the free circle action by $N$ and we denote by $p$ the projection map $M\to N$.
Note that   we have the Gysin sequence
\[ \cdots \to H^{0}(N) \xrightarrow{\cup \hspace*{1pt} e} H^2(N) \xrightarrow{p^{*}}
H^2(M) \xrightarrow{p_{*}} H^{1}(N) \xrightarrow{\cup \hspace*{1pt} e} H^3(N) \to \cdots
\]
where $p_*$ is the map given by `integration along the fiber'.
An easy argument shows that  the Gysin sequence combined with the assumption that $e$ is non-torsion implies that
 $b_{2}^{+}(M) = b_1(N)  - 1$. Since $M$ is symplectic we conclude that $b_1(N)>1$.

We will first show that $N$ is a torus bundle over $S^1$.
Given $R=\Z$ or $R=\F_p$ we define the virtual Betti number to be
 \[ vb_1(N;R) = \mbox{sup}\{b_1({\tilde N},R)\,|\, p \colon {\tilde N} \to N \mbox{ is a finite cover} \}. \]
In particular we write $vb_1(N;R)=\infty$ if there exist finite covers with arbitrarily large first $R$-Betti number.

We will now start by showing that the condition $K = 0$ implies that  $vb_1(N) \leq 3$.
In \cite{Ba03} Baldridge has determined the relation between the Seiberg--Witten invariants of $M$ and $N$:
The Seiberg-Witten invariant $SW_{M}(\kappa)$ of a class $\kappa = p^{*} \xi \in p^*
H^{2}(N) \subset H^{2}(M)$ is given by the following formula, that combines Corollaries 25 and 27 of \cite{Ba03}, \begin{equation} \label{baldfor} SW_{M}(\kappa)
=\sum_{\xi \in (p^*)^{-1}(\kappa)} SW_N(\xi)\in \Z,  \end{equation}
in particular when $b_{2}^{+}(M) = 1$ it is independent on the chamber in which it was calculated. Moreover, if $b_{2}^{+}(M) > 1$, these are the only basic classes. When the canonical class is trivial,
Taubes' constraints (\cite{Ta94,Ta95}) imply that $SW_{M}(0) = 1$ and that  the class $K = 0$ is the only basic class of $M$ (when $b_{2}^{+}(M) = 1$, applying the usual caveats for this case, this is true in $p^*
H^{2}(N)$, and requires the use of Equation \ref{baldfor} and the symmetry of the Seiberg-Witten invariants of $N$).
 We can  then compute  the sum of the coefficients of the Seiberg-Witten invariant of $N$ as
\begin{equation} \label{sum}  \sum_{\xi \in H^2(N)} SW_{N}(\xi) = \sum_{\kappa \in p^* H^2(N)} SW_{M}(\kappa) =SW_M(0)=1. \end{equation}
Now, for all $3$--manifolds with $b_{1}(N) > 1$, the sum of the coefficients of the Seiberg-Witten polynomial equals, by the formula of Meng and Taubes (see \cite{MT96}), the sum of the coefficients of the Alexander polynomial, and the latter vanishes when $b_{1}(N) > 3$ (see \cite[Section II.5.2 and Theorem IX.2.2]{Tu02}).
Equation (\ref{sum}) requires therefore that $b_{1}(N) \leq 3$. Repeating this argument for all
covers of $N$ (for which the Euler class is necessarily nontorsion and the canonical zero) gives the desired bound on $vb_1(N)$.

We want to show that the condition $vb_1(N)\leq 3$ entails that either $N$ is a torus bundle, or $N$ is hyperbolic.
First note that a generalization of an argument of  \cite{McC01} shows that our assumption that $M$ is symplectic implies that $N$ is irreducible. This argument is  described in detail in \cite{Bow09} and it uses the fact that 3-manifold groups are now known to be residually finite (which is a consequence of the proof of the Geometrization Conjecture, we refer to  \cite{He87} for details).
Recall that we had shown that  $b_1(N)>1$; it now follows that  $N$ is Haken.

We can consider then the JSJ decomposition of $N$: we know that either $N$ has a non--trivial JSJ decomposition, or it is Seifert fibered, or it is hyperbolic.

We now prove that for the first two cases the condition $vb_1(N)\leq 3$ allows us to conclude that $N$ is a torus bundle. In fact, $N$ must contain an incompressible torus $T$. (This is obvious if $N$ has a non--trivial JSJ decomposition, but it is true also when $N$ is
Seifert fibered, as Seifert fibered manifolds without incompressible tori must have $b_1 \leq 1$
(cf. \cite[p.~89ff]{Ja80}).)  It now follows from $vb_1(N)\leq 3$ combined with the results of Kojima and Luecke (see \cite[p.~744]{Ko87} or \cite[Theorem~1.1]{Lu88})
 that  there exists a finite cover $p:\ti{N}\to N$ such that $\ti{N}$ is a torus bundle. We claim that this implies that $N$ is also a torus bundle.
 Indeed, a straightforward Thurston norm argument shows that
any non-zero class in $H^1(\ti{N};\Z)$ corresponds to a torus bundle. Since $b_1(N)\geq 2$ we can pick a non-zero class $\phi\in H^1(N;\Z)$.
We see that $(\ti{N},p^*(\phi))$ is a torus bundle and it is well-known that this implies that $(N,\phi)$ is already a torus bundle. (If $N$ is a torus bundle, all its finite covers are torus bundles, so the condition $vb_1(N) \leq 3$ is satisfied.)

We are left therefore with the cases where $N$ is a torus bundle, or an hyperbolic manifold with $vb_{1}(N) \leq 3$. To complete the proof of the Theorem, it remains to show that the latter case can be excluded. It is widely expected (and verified for the
arithmetic case, see \cite{CLR07}) that hyperbolic manifolds with positive Betti number have
$vb_1(N) = \infty$, so it is quite possible that that case is taken care of by the previous result, but even if we lack a general proof of this fact, we will be able to explicitly rule out that case too.

We start by observing that, as consequence of \cite[Section II.5.2 and Theorem IX.2.4]{Tu02}, if the homology group
$H_1(N,\F_p)$ has rank $b_1(N,\F_p) > 3$, then the sum of the coefficients of the Alexander
polynomial vanishes \textit{mod $p$}. Now a consequence of the Lubotzky alternative
(cf. \cite[Corollary~16.4.18]{LS03} and \cite[Theorem 1.3]{La09}) asserts that if $\pi_1(N)$ is a finitely generated linear
group,  then either $\pi_1(N)$ is virtually soluble or, for any prime $p$ we have $vb_1(N,\F_p) = \infty$.
Since $N$ is hyperbolic and orientable, its fundamental group is of course
linear. By \cite[Theorem 4.5]{EM72} the only (orientable) $3$--manifolds with positive Betti
number and soluble fundamental group are torus bundles, whence the first condition cannot occur.
It follows that for any prime $p$, there exists a cover of $N$ (even regular, by \cite[Theorem
5.1]{La09}) whose Alexander polynomial has sum of coefficients that vanishes \textit{mod $p$}.
This entails that the corresponding cover of $M$ violates the combination of Equation
(\ref{sum}) with \cite{MT96}, hence the possibility of a hyperbolic $N$ is excluded.

(If $N$ is a torus bundle, all its finite covers are torus bundles, so the condition $vb_1(N;\F_p) \leq 3$ is satisfied.)

We now showed  that  $M$ is a principal $S^1$-bundle over a 3-manifold $N$ which is a torus bundle and which satisfies $b_1(N)\geq 2$.
Note that  a torus bundle with $b_1(N) \geq 2$ is also an $S^1$--bundle
over $T^2$ (see e.g. \cite{Hat}). It follows that  $M$ is in fact a $T^2$--bundle over $T^2$.
\end{proof}

\begin{remark}
\bn

\item The proof of Theorem \ref{betti} applies, \textit{mutatis mutandis}, to the product case, making
it unnecessary there as here to use Kronheimer's  refined adjunction inequality or Donaldson's theorem on
the existence of symplectic representatives of (sufficiently high multiples of) the dual of
$[\omega]$, that are used instead in \cite{Kr99} and \cite{Vi03}.
\item Note that the statement of Theorem \ref{betti} covers in fact all symplectic manifolds with \textit{torsion} canonical class (a class that \textit{a priori} could be broader,
when $b_{2}^{+}(M) = 1$, than the case of trivial canonical class). In fact, if $b_{2}^{+}(M) = 1$, $b_1(M) = 2$,
and $K$ is torsion, McDuff and Salamon show in \cite{MS95} that $K$ is in fact  trivial. (This can actually be verified, in the case at hand, using Taubes' constraints and the symmetry of $SW_{N}$.)
So Theorem \ref{betti} covers all symplectic manifolds
with Kodaira dimension $0$, in the notation of \cite{Li06} ($M$ is symplectically minimal, as it
is aspherical).
\en
\end{remark}

We will discuss now the case where the Euler class $e \in H^2(N;\Z)$ is torsion.
In principle, we could proceed as in the previous case, using the results of \cite{Ba03} for the relation between the Seiberg-Witten invariants of $M$ and $N$ in case of torsion Euler class. This does not present particular conceptual difficulties but requires, in the case of $b_2^{+}(M) = b_1(N) = 1$, a detailed bookkeeping of the chamber dependence of the Seiberg-Witten invariants for both $M$ and $N$, that would impose on us a somewhat long detour. Instead of following that path, it is simpler to use a straightforward algebro--topological observation to reduce the problem to the product case treated in \cite{FV08a}.
(The same approach was independently taken by Bowden \cite[Proposition~3]{Bow09}.)
We have the following result:

\begin{theorem} \label{tor}
Let $(M,\omega)$ be a symplectic manifold with trivial canonical class. If $M$ admits a free circle action with orbit space $N$ and torsion Euler class $e \in H^2(N;\Z)$, then $N$ is a torus bundle over a circle. Furthermore, if either  $b_1(N)>1$ or $e=0$, then
  $M$ itself is a torus bundle over a torus.
\end{theorem}

\begin{proof}
If the Euler class $e \in H^2(N)$ is trivial, $M$ is a product $S^1 \times N$. It was shown in  \cite{FV08a} that $N$ fibers over $S^1$, hence $M$ is a torus bundle over a torus.

 So we now assume that $e \neq 0$, so that in particular $\Tor(H_1(N))$ is
non--trivial.
Denote by $\Gamma$ an abelian group isomorphic to
$\Tor(H_1(N))$, and pick a map $H_1(N) \to \Gamma$ such that the induced map
\[ \mbox{Tor}(H_1(N))\to H_1(N)\to \Gamma \] is an isomorphism. Denote by $\gamma\colon \pi_1(N) \to H_1(N)\to \Gamma$
the corresponding map from the fundamental group of $N$.  We claim that, denoting as usual by $\pi \colon N_{\Gamma} \to N$ the associated regular cover, and by $\pi\colon M_{\Gamma} \to M$ its $4$--dimensional counterpart, $M_{\Gamma}$ is the product $S^1 \times N_{\Gamma}$. In fact, we have the  commutative diagram
\[ \xymatrix{ 1\ar[r]&\pi_1(N_\Gamma)\ar[r]\ar[d]&\pi_1(N)\ar[d]\ar[r]&\Gamma\ar[d]\ar[r]&1\\
&H_1(N_\Gamma)\ar[r]&H_1(N)\ar[r]&\Gamma&\\
&\mbox{Tor}(H_1(N_\Gamma)) \ar[u]\ar[r]&\mbox{Tor}(H_1(N))\ar[u]\ar[ur]^\cong.}\]
It follows from
the commutativity, and from the exactness of the top horizontal sequence that the   map $\pi_{*}\colon
\mbox{Tor}(H_1(N_{\Gamma}))\to\mbox{Tor}(H_1(N))$, and hence
\[ (\pi_{*})^{*} \colon
\mbox{Ext}(H_{1}(N),\Z) \to \mbox{Ext}(H_{1}(N_{\Gamma}),\Z)\] is trivial. It now follows from the
naturality of the universal coefficient short exact sequence that the Euler class $e_{\Gamma} =
\pi^{*} e$ of the fibration $p_{\Gamma}\colon M_{\Gamma} \to N_{\Gamma}$ is zero, i.e. $M_\G$ is the product $S^1\times N_\G$.

Now assume that $M$ admits a symplectic structure. Then the manifold $S^1 \times N_{\Gamma}$ inherits a symplectic
structure $\pi^{*}\omega$ with canonical class $\pi^{*}K = 0$. It now follows from  \cite{FV08a} that $N_{\Gamma}$ is a torus bundle. A Thurston norm argument shows that $(N_{\Gamma},\phi_{\Gamma})$ is a torus bundle for any
$\phi_{\Gamma}\in H^1(N_{\Gamma};\Z)$.  Now let $\phi\in H^1(N;\Z)$ be any non-trivial element, which exists since $b_1(N)\geq 1$.
By the above $(N_{\Gamma},\pi^*(\phi))$ is a torus bundle and it is well-known that this implies that $(N,\phi)$ is a torus bundle.
When $e = 0$ obviously and when $b_1(N)>1$ with the same argument as in Theorem \ref{betti} we deduce that $M$ itself is a torus bundle over a torus.
 \end{proof}

 Note that in all cases a manifold $M$ as in the statement of Theorem \ref{tor} is finitely covered by a torus bundle over a torus.

As for the case of Theorem \ref{betti}, \cite{MS95} guarantees that the manifolds above are the only symplectic manifolds with Kodaira dimension zero.

We finish this note with one observation and one conjecture. The proof of Theorem \ref{lub} boils down, using the appropriate topological arguments, to the study of virtual Betti numbers of symplectic manifolds with $K = 0$. This approach, in various forms, appears in \cite{Bau08,Bow09,FV08a,Li06b}. From our viewpoint, the upper bounds on the virtual Betti number are a consequence, via dimensional reduction, of known constraints on the sum of the coefficients of Alexander polynomials, while in the approach of \cite{Bau08,Li06b} (that applies, differently from ours, to any $4$--manifold) it follows from properties of the Bauer-Furuta invariants. Overlapping the results available with both approaches, it seems sensible to conjecture that any symplectic $4$-manifold $M$ with $K = 0$ should satisfy $vb_1(M,\F_{p}) \leq 4$. Verifying this conjecture could give further evidence that the list of \cite{Li06} is complete.



\end{document}